\documentclass[12pt,lenq]{amsart}
\usepackage{amsmath}
\usepackage{amscd}
\usepackage{amssymb}
\usepackage{amsbsy}
\usepackage{amsfonts}
\usepackage{latexsym}
\usepackage{graphics}
\usepackage{amsmath,amscd,latexsym}
\usepackage{multirow}
\usepackage{array}
\usepackage{paralist}
\usepackage{titletoc}

\theoremstyle{plain}
\newtheorem{theorem}{Theorem}[section]
\newtheorem{proposition}[theorem]{Proposition}
\newtheorem{lemma}[theorem]{Lemma}
\newtheorem{corollary}[theorem]{Corollary}
\newtheorem{remark}[theorem]{Remark}
\newtheorem{definition}[theorem]{Definition}

\newtheorem{main theorem}[theorem]{Main Theorem}

\newtheorem{convention}[theorem]{Convention}

\newcommand{\interior}{\operatorname{int}}

\newcommand{\NN}{\mathbb{N}}
\newcommand{\ZZ}{\mathbb{Z}}
\newcommand{\QQ}{\mathbb{Q}}
\newcommand{\RR}{\mathbb{R}}
\newcommand{\CC}{\mathbb{C}}
\newcommand{\HH}{\mathbb{H}}
\newcommand{\QQQ}{\hat{\mathbb{Q}}}

\newcommand{\Conway}{\mbox{\boldmath$S$}^{2}}
\newcommand{\Conways}
{(\mbox{\boldmath$S$}^{2},\mbox{\boldmath$P$})}
\newcommand{\PP}{\mbox{\boldmath$P$}}
\newcommand{\PConway}{\mbox{\boldmath$S$}}

\newcommand{\OO}{\mbox{\boldmath$O$}}

\newcommand{\rtangle}[1]{(B^3,t({#1}))}

\newcommand{\DD}{\mathcal{D}}
\newcommand{\RGPC}[2]{\Gamma({#1};{#2})}

\newcommand{\RGP}[1]{\Gamma_{#1}}

\newcommand{\Hecke}{\mbox{$G$}}
\newcommand{\orbo}{\mbox{\boldmath$O$}}
\newcommand{\orbs}{\mbox{\boldmath$S$}}
\newcommand{\orbb}{\mbox{\boldmath$B$}}

\newcommand{\cfr}{\mbox{\boldmath$a$}}

\newcommand{\svert}{\,|\,}

\newcommand{\llangle}{\langle\langle}
\newcommand{\rrangle}{\rangle\rangle}

\begin{document}

\title{Epimorphisms from 2-bridge link groups onto Heckoid groups (I)}

\author{Donghi Lee}
\address{Department of Mathematics\\
Pusan National University \\
San-30 Jangjeon-Dong, Geumjung-Gu, Pusan, 609-735, Republic of Korea}
\email{donghi@pusan.ac.kr}

\author{Makoto Sakuma}
\address{Department of Mathematics\\
Graduate School of Science\\
Hiroshima University\\
Higashi-Hiroshima, 739-8526, Japan}
\email{sakuma@math.sci.hiroshima-u.ac.jp}

\subjclass[2010]{Primary 57M25, 57M50 \\
\indent {The first author was supported by Basic Science Research Program
through the National Research Foundation of Korea(NRF) funded
by the Ministry of Education, Science and Technology(2012R1A1A3009996).
The second author was supported
by JSPS Grants-in-Aid 22340013.}}

\begin{abstract}
Riley ``defined'' the Heckoid groups for $2$-bridge links
as Kleinian groups, with nontrivial torsion,
generated by two parabolic transformations,
and he constructed an infinite family of
epimorphisms from $2$-bridge link groups onto Heckoid groups.
In this paper, we make Riley's definition explicit,
and give a systematic construction
of epimorphisms from $2$-bridge link groups onto
Heckoid groups, generalizing Riley's construction.
\end{abstract}
\maketitle

\begin{center}
{\it In honour of J. Hyam Rubinstein
and his contribution to mathematics}
\end{center}

\section{Introduction}

In \cite{Riley2},
Riley introduced an infinite collection of Laurent polynomials,
called the Heckoid polynomials,
associated with a $2$-bridge link $K$,
and observed, through extensive computer experiments,
that these Heckoid polynomials define the affine representation variety
of certain groups, the Heckoid groups for $K$.
To be more precise, he ``defines''
the Heckoid group of index $q\ge 3$ for $K$
to be a Kleinian group generated by two parabolic transformations
which are obtained by choosing a ``right'' root of the Heckoid polynomials
(see \cite[the paragraph following Theorem ~A in p.390]{Riley2}).
The classical Hecke groups,
introduced in \cite{Hecke}, are essentially the simplest Heckoid groups.
Riley discussed relations of the Heckoid polynomials with the polynomials
defining the nonabelian $SL(2,\CC)$-representations
of $2$-bridge link groups introduced in \cite{Riley1},
and proved that each Heckoid polynomial divides
the nonabelian representation polynomials of $2$-bridge links $\tilde K$,
where $\tilde K$ belongs to an infinite collection
of $2$-bridge links determined by $K$ and the index $q$.
This suggests that there are epimorphisms from the link group of $\tilde K$
onto the Heckoid group of index $q$ for $K$,
as observed in \cite[the paragraph following Theorem ~B in p.391]{Riley2}.

The purpose of this paper is
(i) to give an explicit combinatorial definition of the Heckoid groups
for $2$-bridge links (Definition ~\ref{def:odd_Heckoid_group}),
(ii) to prove that the Heckoid groups are identified with
Kleinian groups generated by two parabolic transformations
(Theorem ~\ref{thm.Kleinian_heckoid}),
and
(iii) to give a systematic construction
of epimorphisms from $2$-bridge link groups onto
Heckoid groups, generalizing Riley's construction
(Theorem ~\ref{thm:epimorophism} and Remark ~\ref{rem:Riley's_theorem}).

We note that the results (i) and (ii) are essentially contained in
the work of Agol ~\cite{Agol}, in which he announces a complete
classification of the non-free Kleinian groups generated by two-parabolic transformations.
Moreover, this classification theorem gives a nice characterization of
the Heckoid groups, by showing that they are exactly the Kleinian groups,
with nontrivial torsion,
generated by two-parabolic transformations.

The result (iii) is an analogy of
the systematic construction of epimorphisms
between $2$-bridge link groups given in
\cite[Theorem ~1.1]{Ohtsuki-Riley-Sakuma}.
In the sequel \cite{lee_sakuma_7} of this paper,
we prove, by using small cancellation theory, that
the epimorphisms in Theorem ~\ref{thm:epimorophism}
are the only upper-meridian-pair-preserving epimorphisms
from $2$-bridge link groups onto {\it even} Heckoid groups.
This in turn forms
an analogy of \cite[Main Theorem ~2.4]{lee_sakuma},
which gives a complete characterization
of upper-meridian-pair-preserving epimorphisms
between $2$-bridge link groups.

This paper is organized as follows.
In Section ~\ref{statements},
we describe the main results.
In Section ~\ref{fricke_surfaces},
we give an explicit combinatorial definition of Heckoid groups.
Sections ~\ref{sec:proof_if_part}, \ref{sec:topological-description} and
\ref{sec:Kleinian-Heckod-groups}, respectively, are devoted to
the proof of Theorem ~\ref{thm:epimorophism},
the topological description of Heckoid orbifolds,
and the proof of Theorem ~\ref{thm.Kleinian_heckoid}.

Throughout this paper,
we denote the orbifold fundamental group of an orbifold $X$
by $\pi_1(X)$.

\section{Main results}
\label{statements}

Consider the discrete group, $H$, of isometries
of the Euclidean plane $\RR^2$
generated by the $\pi$-rotations around
the points in the lattice $\ZZ^2$.
Set $\Conways:=(\RR^2,\ZZ^2)/H$
and call it the {\it Conway sphere}.
Then $\Conway$ is homeomorphic to the 2-sphere,
and $\PP$ consists of four points in $\Conway$.
We also call $\Conway$ the Conway sphere.
Let $\PConway:=\Conway-\PP$ be the complementary
4-times punctured sphere.
For each $s \in \QQQ:=\QQ\cup\{\infty\}$,
let $\alpha_s$ be the simple loop in $\PConway$
obtained as the projection of a line in $\RR^2-\ZZ^2$
of slope $s$.
Then $\alpha_s$ is {\it essential} in $\PConway$,
i.e., it does not bound a disk in $\PConway$
and is not homotopic to a loop around a puncture.
Conversely, any essential simple loop in $\PConway$
is isotopic to $\alpha_s$ for
a unique $s\in\QQQ$.
Then $s$ is called the {\it slope} of the simple loop.
We abuse notation to denote by $\alpha_s$
the pair of conjugacy classes in $\pi_1(\PConway)$
represented by the loop $\alpha_s$ with two possible orientations.

A {\it trivial tangle} is a pair $(B^3,t)$,
where $B^3$ is a 3-ball and $t$ is a union of two
arcs properly embedded in $B^3$
which is simultaneously
parallel to a union of two
mutually disjoint arcs in $\partial B^3$.
Let $\tau$ be the simple unknotted arc in $B^3$
joining the two components of $t$
as illustrated in Figure ~\ref{fig.trivial-tangle}.
We call it the {\it core tunnel} of the trivial tangle.
Pick a base point $x_0$ in $\interior \tau$,
and let $(\mu_1,\mu_2)$ be the generating pair
of the fundamental group $\pi_1(B^3-t,x_0)$
each of which is represented by
a based loop consisting of
a small peripheral simple loop
around a component of $t$ and
a subarc of $\tau$ joining the circle to $x_0$.
For any base point $x\in B^3-t$,
the generating pair of $\pi_1(B^3-t,x)$
corresponding to the generating pair $(\mu_1,\mu_2)$
of $\pi_1(B^3-t,x_0)$ via a path joining $x$ to $x_0$
is denoted by the same symbol.
The pair $(\mu_1,\mu_2)$ is unique up to
(i) reversal of the order,
(ii) replacement of one of the members with its inverse,
and (iii) simultaneous conjugation.
We call the equivalence class of $(\mu_1,\mu_2)$
the {\it meridian pair} of $\pi_1(B^3-t)$.

\begin{figure}
\begin{center}
\includegraphics{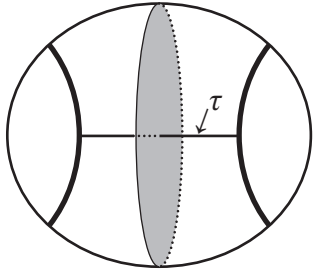}
\end{center}
\caption{\label{fig.trivial-tangle}
A trivial tangle}
\end{figure}

By a {\it rational tangle},
we mean a trivial tangle $(B^3,t)$
which is endowed with a homeomorphism from
$\partial(B^3,t)$ to $\Conways$.
Through the homeomorphism we identify
the boundary of a rational tangle with the Conway sphere.
Thus the slope of an essential simple loop in
$\partial B^3-t$ is defined.
We define
the {\it slope} of a rational tangle
to be the slope of
an essential loop on $\partial B^3 -t$ which bounds a disk in $B^3$
separating the components of $t$.
(Such a loop is unique up to isotopy
on $\partial B^3 -t$ and is called a {\it meridian}
of the rational tangle.)
We denote a rational tangle of slope $r$ by
$\rtangle{r}$.
By van Kampen's theorem, the fundamental group
$\pi_1(B^3-t(r))$ is identified
with the quotient
$\pi_1(\PConway)/\llangle\alpha_r \rrangle$,
where $\llangle\alpha_r \rrangle$ denotes the normal
closure.

For each $r\in \QQQ$,
the {\it 2-bridge link $K(r)$ of slope $r$}
is defined to be the sum of the rational tangles of slopes
$\infty$ and $r$, namely,
$(S^3,K(r))$ is
obtained from $\rtangle{\infty}$ and
$\rtangle{r}$
by identifying their boundaries through the
identity map on the Conway sphere
$\Conways$. (Recall that the boundaries of
rational tangles are identified with the Conway sphere.)
$K(r)$ has one or two components according as
the denominator of $r$ is odd or even.
We call $\rtangle{\infty}$
and $\rtangle{r}$, respectively,
the {\it upper tangle} and {\it lower tangle}
of the 2-bridge link.
By van Kampen's theorem, the link group $G(K(r))=\pi_1(S^3-K(r))$ is obtained as follows:
\[
G(K(r))=\pi_1(S^3-K(r))
\cong \pi_1(\PConway)/ \llangle\alpha_{\infty},\alpha_r\rrangle
\cong \pi_1(B^3-t(\infty))/\llangle\alpha_r\rrangle.
\]
We call the image in the link group
of the meridian pair of $\pi_1(B^3-t(\infty))$
(resp. $\pi_1(B^3-t(r))$)
the {\it upper meridian pair} (resp. {\it lower meridian pair}).

\begin{figure}[h]
\begin{center}
\includegraphics{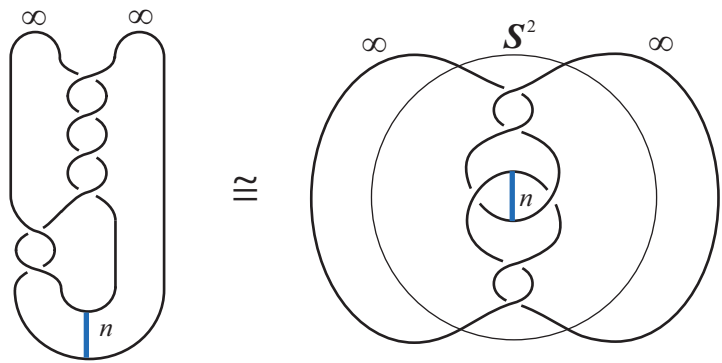}
\end{center}
\caption{
\label{fig.Hekoid-orbifold}
The even Heckoid orbifold $\orbs(r;n)$
of index $n$ for the $2$-bridge link $K(r)$,
where we employ Convention ~\ref{conv:orbifold}.
Here $(S^3,K(r))=(B^3,t(\infty))\cup (B^3,t(r))$ is the $2$-bridge link
with $r=2/9=[4,2]$ (with a single component).
The rational tangles $(B^3,t(\infty))$ and $(B^3,t(r))$, respectively,
are the outside and the inside of the bridge sphere $\Conway$.
}
\end{figure}

For a rational number $r$ ($\ne \infty$)
and an integer $n\ge 2$,
the ({\it even}) {\it Heckoid orbifold, $\orbs(r;n)$,
of index $n$ for the $2$-bridge link $K(r)$}
is the $3$-orbifold as shown in Figure ~\ref{fig.Hekoid-orbifold}.
Namely, the underlying space $|\orbs(r;n)|$ is $E(K(r))$
and the singular set is the {\it lower} tunnel,
where the index of singularity is $n$.
Here, the lower tunnel means the core tunnel of the lower tangle.
the core tunnel
The ({\it even}) {\it Hekoid group} $\Hecke(r;n)$
is defined to be the orbifold fundamental group $\pi_1(\orbs(r;n))$.
By van Kampen's theorem for orbifold fundamental groups
(cf. \cite[Corollary ~2.3]{Boileau-Porti}),
we have
\[
\Hecke(r;n)
\cong\pi_1(\PConway)/ \llangle\alpha_{\infty},\alpha_r^n\rrangle
\cong \pi_1(B^3-t(\infty))/\llangle\alpha_r^n\rrangle.
\]
In particular,
the even Heckoid group $\Hecke(r;n)$ is a two-generator and one-relator group.
We call the image in $\Hecke(r;n)$
of the meridian pair of $\pi_1(B^3-t(\infty))$
the {\it upper meridian pair}.

The announcement by Agol ~\cite{Agol}
and the announcement
made in the second author's joint work with
Akiyoshi, Wada and Yamashita ~\cite[Section ~3 of Preface]{ASWY}
(cf. Remark ~\ref{rem:ASWY})
suggest that the group $\Hecke(r;n)$ makes sense
even when $n$ is a half-integer greater than $1$.
The precise definition of $\Hecke(r;n)$ with $n>1$ a half-integer
is given in Definition ~\ref{def:odd_Heckoid_group},
and a topological description of the corresponding orbifold, $\orbs(r;n)$,
is given by Proposition ~\ref{prop:odd-Heckoid-orbifold}
(see Figures ~\ref{odd-Heckoid-orbifold1} and ~\ref{odd-Heckoid-orbifold2}).
When $n>1$ is a non-integral half-integer,
$\Hecke(r;n)$ and $\orbs(r;n)$, respectively, are called
the ({\it odd}) {\it Heckoid orbifold} and
the ({\it odd}) {\it Heckoid group} of index $n$ for $K(r)$.
There is a natural epimorphism from $\pi_1(B^3-t(\infty))$
onto the odd Heckoid group $\Hecke(r;n)$,
and the image of the meridian pair of $\pi_1(B^3-t(\infty))$
is called the
{\it upper meridian pair} of $\Hecke(r;n)$.
Thus the odd Heckoid groups are also two-generator groups.
However, we show that
they are not one-relator groups (Proposition ~\ref{prop-not-one-relator}).

\begin{remark}
{
\rm
Our terminology is slightly different from that of \cite{Riley2},
where $\Hecke(r;n)$ is called the Heckoid group of index ``$2n$'' for $K(r)$.
The Heckoid orbifold $\orbs(r;n)$ and the Heckoid group $\Hecke(r;n)$
are {\it even} or {\it odd} according to whether Riley's index $2n$
is even or odd.
}
\end{remark}

We prove the following theorem,
which was anticipated in \cite{Riley2}
and is contained in \cite{Agol} without proof.

\begin{theorem}
\label{thm.Kleinian_heckoid}
For $r$ a rational number and $n>1$ an integer or a half-integer,
the Heckoid group $\Hecke(r;n)$ is isomorphic to a
geometrically finite Kleinian group
generated by two parabolic transformations.
\end{theorem}

In order to explain a systematic construction of epimorphisms
from $2$-bridge link groups onto Heckoid groups,
we prepare a few notation.
Let $\DD$ be the {\it Farey tessellation}, that is,
the tessellation of the upper half
space $\HH^2$ by ideal triangles which are obtained
from the ideal triangle with the ideal vertices $0, 1, \infty \in \QQQ$
by repeated reflection in the edges.
Then $\QQQ$ is identified with the set of the ideal vertices of $\DD$.
For each $r\in \QQQ$, let $\RGP{r}$ be the group of automorphisms of
$\DD$ generated by reflections in the edges of $\DD$ with an endpoint $r$.
Let $n>1$ be an integer or a half-integer,
and let $C_r(2n)$ be the group of automorphisms of $\DD$ generated
by the parabolic transformation, centered on the vertex $r$,
by $2n$ units in the clockwise direction.

For $r$ a rational number
and $n>1$ an integer or a half-integer,
let $\RGPC{r}{n}$ be the group generated by $\RGP{\infty}$ and $C_r(2n)$.
Then we have the following systematic construction
of epimorphisms from $2$-bridge link groups onto Heckoid groups.

\begin{theorem}
\label{thm:epimorophism}
Suppose that $r$ is a rational number and
that $n>1$ is an integer or a half-integer.
For $s\in\QQQ$, if $s$ or $s+1$ belongs to the $\RGPC{r}{n}$-orbit of $\infty$,
then there is an upper-meridian-pair-preserving epimorphism
from $G(K(s))$ to $\Hecke(r;n)$.
\end{theorem}

This theorem may be regarded as a generalization of
Theorem ~B and Theorem ~3 of Riley ~\cite{Riley2}.
In fact, they correspond to the case
when $s$ belongs to the orbit of $\infty$
by the infinite cyclic subgroup $C_r(2n)$ of $\RGPC{r}{n}$
(see Remark ~\ref{rem:Riley's_theorem}).

The above theorem is actually obtained from the following theorem.

\begin{theorem}
\label{thm:fundametal_domain}
Suppose that $r$ is a rational number and
that $n>1$ is an integer or a half-integer.
Let $s$ and $s'$ be elements of $\QQQ$
which belong to the same $\RGPC{r}{n}$-orbit.
Then the conjugacy classes
$\alpha_s$ and $\alpha_{s'}$ in $\Hecke(r;n)$ are equal.
In particular, if $s$ belongs to the $\RGPC{r}{n}$-orbit of $\infty$,
then $\alpha_s$ is the trivial conjugacy class
in $\Hecke(r;n)$.
\end{theorem}

\section{Combinatorial definition of Heckoid groups}
\label{fricke_surfaces}

In this section, we give an explicit combinatorial definition
of even/odd Heckoid groups.
Consider the $(2,2,2,\infty)$-orbifold,
$\OO:=(\RR^2-\ZZ^2)/\hat H$,
where $\hat H$ is the group generated by $\pi$-rotations
around the points in $(\frac{1}{2}\ZZ)^2$.
Note that $\OO$ has a once-punctured sphere as the underlying space,
and has three cone points of cone angle $\pi$.
The orbifold fundamental group of $\OO$ has the presentation
\[
\pi_1(\OO)=\langle P,Q,R \svert P^2=Q^2=R^2=1 \rangle,
\]
where $D:=(PQR)^{-1}$ is represented by the puncture of $\OO$
(see Figure ~\ref{fig.(2,2,2,infty)-orbifold}).
For each $s\in\QQQ$, the image of a straight line of slope $s$
in $\RR^2-\ZZ^2$ disjoint from the singular set of $\hat H$
projects to a simple loop, $\beta_s$, in $\OO$ disjoint from
the cone points.
Thus the loop $\beta_s$ (with an orientation) represents a conjugacy class in $\pi_1(\OO)$.
We abuse notation to denote by $\beta_s$ the pair of
the conjugacy classes in $\pi_1(\OO)$
represented by $\beta_s$ with two possible orientations.
Throughout this paper,
we choose the generating set $\{P,Q,R\}$ of $\pi_1(\OO)$
so that the conjugacy classes
$\beta_0$ and $\beta_{\infty}$ are represented by $RQ$ and $PQ$, respectively
(see Figure ~\ref{fig.(2,2,2,infty)-orbifold} and \cite[Section ~2.1]{ASWY}).

\begin{figure}[h]
\begin{center}
\includegraphics{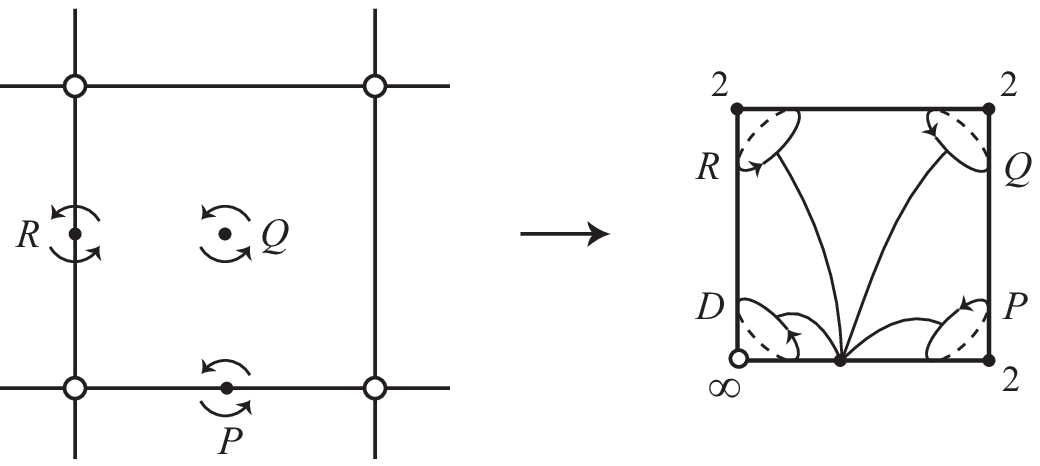}
\end{center}
\caption{
\label{fig.(2,2,2,infty)-orbifold}
The orbifold $\OO$}
\end{figure}

The Conway sphere $\PConway=(\RR^2-\ZZ^2)/H$
is the $(\ZZ/2\ZZ)^2$-covering of $\OO$,
and hence $\pi_1(\PConway)$ is a normal subgroup of $\pi_1(\OO)$
such that $\pi_1(\OO)/\pi_1(\PConway)\cong (\ZZ/2\ZZ)^2$.
Each simple loop $\alpha_s$ in $\PConway$
doubly covers the simple loop $\beta_s$, and so
we have $\alpha_s=\beta_s^2$ as conjugacy classes in $\pi_1(\OO)$.

For each $r\in\QQQ$ and integer $m\ge 2$,
consider the orbifold, $\orbb(r;m)$, as illustrated in Figure ~\ref{fig.ball-orbifold}.
In order to give its explicit description,
we prepare notation following \cite{Mecchia-Zimmermann}.
For an integer $m \ge 2$, let $D^2(m)$ be the {\it discal $2$-orbifold}
obtained from the unit disk $D^2$ in the complex plane
by taking the quotient of the action generated by
the $2\pi/m$-rotation $z\mapsto e^{2\pi i/m}z$.
We call the product $3$-orbifold $D^2(m)\times I$ with $I=[0,1]$
a {\it $2$-handle orbifold}.
The quotient orbifold of the unit $3$-ball $B^3$ in $\RR^3$
by the dihedral subgroup, $D_{2m}$, of $SO(3)$ of order
$2m$ is denoted by $B^3(2,2,m)$ and is called a {\it $3$-handle orbifold}.
By using this notation,
the orbifold $\orbb(r;m)$ has the following description (see Figure ~\ref{fig.ball-orbifold}).
Let $\check\OO$ be the compact $2$-orbifold obtained from $\OO$
by removing an open regular neighborhood of the puncture.
Then $\orbb(r;m)$ is obtained from
the product orbifold $\check \OO\times[0,1]$ by attaching $2$- and $3$-handle
orbifolds as follows.
\begin{enumerate}[\indent \rm (1)]
\item
Attach a $2$-handle orbifold $D^2(m)\times I$
along the simple loop $\beta_r\times \{0\}$,
i.e.,
identify $(\partial D^2(m))\times I$
with an annular neighborhood of
$\beta_r\times \{0\}$ in the boundary of $\check\OO\times I$.

\item
Cap off the spherical orbifold boundary of the resulting orbifold
by a $3$-handle orbifold $B^3(2,2,m)$.
\end{enumerate}
Note that the $2$-dimensional orbifold $\check\OO$ sits in
the boundary of $\orbb(r;m)$; we call it the
{\it outer boundary} of $\orbb(r;m)$,
and denote it by $\partial_{out}\orbb(r;m)$.
To be precise, as in the definition of rational tangles,
$\orbb(r;m)$ is defined to be the orbifold as in Figure ~\ref{fig.ball-orbifold}
which is endowed with a homeomorphism
from $\partial_{out}\orbb(r;m)$ to $\check\OO$.
Thus, by van Kampen's theorem
for orbifold fundamental groups \cite[Corollary ~2.3]{Boileau-Maillot-Porti},
we can identify the orbifold fundamental group
$\pi_1(\orbb(r;m))$ with
$\pi_1(\check\OO)/\llangle \beta_r^m\rrangle=\pi_1(\OO)/\llangle \beta_r^m\rrangle$.
(Here we use the fact that the inclusion map
$\partial B^3(2,2,m)\to B^3(2,2,m)$
induces an isomorphism between
the orbifold fundamental groups.)

\begin{figure}[h]
\begin{center}
\includegraphics{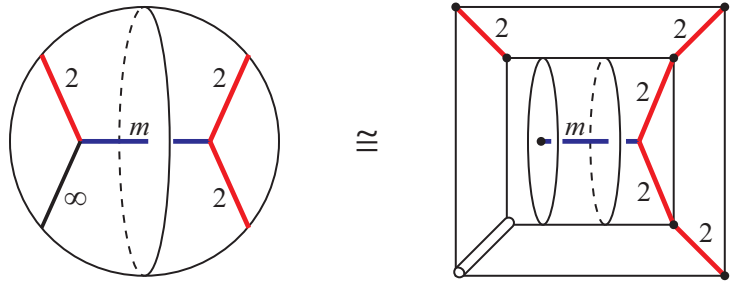}
\end{center}
\caption{
\label{fig.ball-orbifold}
The orbifold $\orbb(r;m)=(\check\OO\times I)\cup (D^2(m)\times I)\cup B^3(2,2,m)$ with $r=\infty$,
where we employ Convention ~\ref{conv:orbifold}.
}
\end{figure}

For a rational number $r$ and an integer $m\ge 2$,
let $\orbo(r;m)$ be the orbifold obtained by identifying
$\orbb(\infty;2)$ and $\orbb(r;m)$, along their outer boundaries,
via their identification with $\check\OO$.
By van Kampen's theorem, the orbifold fundamental group of $\orbo(r;m)$ is given by
the following formula:
\[
\pi_1(\orbo(r;m)) \cong
\pi_1(\OO)/\llangle \beta_{\infty}^2, \beta_r^m\rrangle.
\]

\begin{proposition}
\label{prop:even_Heckoid_group}
For a rational number $r$ and an integer $n>1$,
the even Heckoid orbifold $\orbs(r;n)$
is a $(\ZZ/2\ZZ)^2$-covering of $\orbo(r;m)$, where $m=2n$.
In particular, the even Heckoid group $\Hecke(r;n)$
is identified with the image of
the homomorphism, $\psi$, which is the following composition
of two natural homomorphisms
\[
\pi_1(\PConway) \to \pi_1(\OO) \to \pi_1(\orbo(r;m)).
\]
\end{proposition}

\begin{proof}
{\rm
Let $\check \PConway$ be the compact $2$-manifold obtained from
$\PConway$ by removing open regular neighborhoods of the punctures.
Then we see that
the even Heckoid orbifold $\orbs(r;n)$ is obtained from
$\check\PConway\times [-1,1]$
by attaching a $2$-handle $D^2\times I$ along $\alpha_{\infty}\times \{1\}$
and by attaching a $2$-handle orbifold $D^2(n)\times I$
along $\alpha_{r}\times \{-1\}$.
Note that the group $\hat H/H\cong (\ZZ/2\ZZ)^2$
acts on $\check\PConway$ and the quotient is identified with $\check\OO$.
Since the loops $\alpha_{\infty}$ and $\alpha_{r}$
on $\check\PConway$ can be chosen so that
they are invariant by the action,
it extends to an action on $\orbs(r;n)$.
Moreover the quotient of
$\check\PConway\times [0,1]\cup D^2\times I$
and that of $\check\PConway\times [-1,0]\cup D^2(n)\times I$
are identified with $\orbb(\infty;2)$ and $\orbb(r;m)$, respectively.
Hence $\orbs(r;n)$
is a $(\ZZ/2\ZZ)^2$-covering of $\orbo(r;m)$.
Since the covering $\orbs(r;n)\to \orbo(r;m)$
is ``induced'' by the covering $\PConway\to\orbo$,
and since the natural homomorphism
$\pi_1(\PConway)\to \pi_1(\orbs(r;n))$ is surjective,
we see that $\Hecke(r;n)$
is identified with $\mathrm{Im}(\psi)$.
}
\end{proof}

This motivates us to introduce the following definition.

\begin{definition}
\label{def:odd_Heckoid_group}
{\rm
For a rational number $r$ and
a non-integral half-integer $n$ greater than $1$,
the ({\it odd}) {\it Heckoid group $\Hecke(r;n)$ of index $n$ for $K(r)$}
is defined to be the image, $\mathrm{Im}(\psi)$, of the natural map
\[
\psi: \pi_1(\PConway) \to \pi_1(\OO) \to \pi_1(\orbo(r;m)),
\]
where $m=2n$.
The covering orbifold of $\orbo(r;m)$ corresponding to the subgroup
$\Hecke(r;n)$ of $\pi_1(\orbo(r;m))$
is denoted by $\orbs(r;n)$ and is called the ({\it odd}) {\it Heckoid orbifold
for the $2$-bridge link $K(r)$ of index $n$}.
(See Section ~\ref{sec:topological-description} for a topological description of this orbifold.)
}
\end{definition}

Note that $\psi$ is equal to the composition
\[
\pi_1(\PConway)
\to
\pi_1(B^3-t(\infty))=
\pi_1(\PConway)/ \llangle\alpha_{\infty}\rrangle
\to
\pi_1(\OO)/\llangle \beta_{\infty}^2\rrangle
\to
\pi_1(\orbo(r;m)).
\]
Since $\pi_1(\PConway)\to\pi_1(B^3-t(\infty))$ is surjective and since
$\pi_1(B^3-t(\infty))$ is a free group of rank $2$,
the Heckoid group $\Hecke(r;n)$ is generated by two elements.
However, no odd Heckoid group is a one-relator group
(see Proposition ~\ref{prop-not-one-relator}).

\section{Proofs of Theorems ~\ref{thm:epimorophism} and ~\ref{thm:fundametal_domain}}
\label{sec:proof_if_part}

The following lemma, on the existence of certain self-homeomorphisms of the orbifold $\orbb(r;m)$,
is the heart of Theorem ~\ref{thm:fundametal_domain}.
For the definition of a homeomorphism (diffeomorphism) between orbifolds,
see \cite[Section ~2.1.3]{Boileau-Maillot-Porti} or \cite[p.138]{Kapovich}.

\begin{lemma}
\label{lem:orbimap1}
{\rm (1)}
For $r \in \QQQ$ and an integer $m \ge 2$,
let $F$ be a discal $2$-suborbifold properly embedded in $\orbb(r;m)$ bounded by $\beta_r$,
and let $\varphi$ be the $m$-th power of the Dehn twist
of the underlying space $|\orbb(r;m)|$, preserving the singular set,
along the disk $|F|$.
Then $\varphi$ is a self-homeomorphism of the orbifold $\orbb(r;m)$,
which induces the identity {\rm (}outer{\rm )} automorphism of $\pi_1(\orbb(r;m))$.

{\rm (2)}
For an integer $m \ge 2$,
let $\gamma$ be the reflection of $|\orbb(\infty;m)|$ of Figure ~\ref{fig.ball-orbifold}
in the sheet of the figure.
Then $\gamma$ is a self-homeomorphism of the orbifold $\orbb(\infty;m)$.
Moreover, if $m=2$, then
$\gamma$ induces the identity {\rm (}outer{\rm )} automorphism of $\pi_1(\orbb(\infty;m))$.
\end{lemma}

\begin{proof}
{\rm
(1) To show the first assertion, we have only to check that
each singular point $x$ of $\orbb(r;m)$ has a neighborhood, $U_x$,
such that the restriction of $\varphi$ to $U_x$
lifts to an equivariant homeomorphism from
a manifold covering of $U_x$ to that of $\varphi(U_x)$.
But, this follows from the following observation.
Let $p:D^2\times [0,1] \to D^2(m)\times [0,1]$
be the universal covering of the $2$-handle orbifold,
given by $p(z,t)=(z^m,t)$,
where we identify both $D^2$ and $|D^2(m)|$
with the unit disk in the complex plane.
Let $\varphi$ be the $m$-th power of the Dehn twist of $|D^2(m)\times [0,1]|$
given by $\varphi(z,t)=(e^{2\pi mti}z,t)$.
Then it is covered by the Dehn twist, $\tilde\varphi$, of $D^2\times [0,1]$,
defined by $\tilde\varphi(z,t)=(e^{2\pi ti}z,t)$,
namely we have $\varphi\circ p = p\circ \tilde\varphi$.
(The corresponding automorphism of the local group, $\ZZ/m\ZZ$, is the identity map.)
Thus we have shown the first assertion
that $\varphi$ is a self-homeomorphism of the orbifold $\orbb(r;m)$.

To show the second assertion, we may assume $r=\infty$ without loss of generality.
Then we can see by using Figure ~\ref{fig.(2,2,2,infty)-orbifold} that
\[
(\varphi_*(P), \varphi_*(Q), \varphi_*(R))=
(\beta_{\infty}^m P \beta_{\infty}^{-m}, \beta_{\infty}^m Q \beta_{\infty}^{-m},R),
\]
where
$\beta_{\infty}=PQ\in\pi_1(\orbb(\infty;m))$.
Since $\beta_{\infty}^m=1$ in $\pi_1(\orbb(\infty;m))$,
we see that $\varphi_*$ is the identity map.

(2) We show that
$\gamma$ satisfies the local condition
(in the definition of a homeomorphism between orbifolds)
at every singular point $x$.
Suppose first that $x$ is contained in the interior of an edge of the singular set.
Then $x$ has a neighborhood homeomorphic to the $2$-handle orbifold $D^2(m)\times [0,1]$
such that the restriction of $\gamma$ to it is given by $\gamma(z,t)=(\bar z, t)$.
This is covered by the self-homeomorphism $\tilde\gamma$ of
the universal cover $D^2\times [0,1]$,
defined by $\tilde\gamma(z,t)=(\bar z,t)$.
(The corresponding automorphism of the local group, $\ZZ/m\ZZ$, of $x$
is given by $[k]\mapsto [-k]$ for every $[k]\in \ZZ/m\ZZ$.)
Suppose next that $x$ is the vertex,
on which the edges of indices $2$, $2$, and $m$ are incident.
Then $x$ has a neighborhood homeomorphic to
the $3$-handle orbifold $B^3(2,2,m)=B^3/D_{2m}$.
Then the restriction of the map $\gamma$ is covered by
the reflection in the disk in $B^3$ containing the axes of
the pair of order $2$ generators of $D_{2m}$.
(The corresponding automorphism of the local group, $D_{2m}$, of $x$
is the identity map.)
Thus we have shown that $\gamma$ is
a self-homeomorphism of the orbifold $\orbb(\infty;m)$.

To show the second assertion, observe by using Figure ~\ref{fig.(2,2,2,infty)-orbifold} that
\[
(\gamma_*(P), \gamma_*(Q), \gamma_*(R))=
(P, \beta_{\infty}Q\beta_{\infty}^{-1},\beta_{\infty}R\beta_{\infty}^{-1}).
\]
By composing the inner automorphism
$\iota:x\mapsto \beta_{\infty}^{-1}x\beta_{\infty}$, we have
\[
(\iota\circ\gamma_*(P), \iota\circ\gamma_*(Q), \iota\circ\gamma_*(R))=
(\beta_{\infty}^{-1}P\beta_{\infty}, Q,R)
=(QPQ,Q,R).
\]
If $m=2$, then $(PQ)^2=\beta_{\infty}^2=1$ and so
$\iota\circ\gamma_*(P)=Q(PQ)=Q(QP)=P$.
Hence $\gamma_*$ is the identity outer automorphism when $m=2$.
}
\end{proof}

We can easily observe that
the restrictions of the homeomorphisms, $\varphi$ and $\gamma$,
to the outer boundary $\check\OO$ act on the set of essential simple loops in $\check\OO$
by the following rule.
\begin{enumerate}[\indent \rm (1)]
\item
$\varphi(\beta_s)=\beta_{A_{(r;m)}(s)}$,
where $A_{(r;m)}$ is the automorphism of
the Farey tessellation $\DD$
which is the parabolic transformation,
centered at the vertex $r$, by $m$ units in the clockwise direction
(i.e., a generator of the infinite cyclic group $C_r(m)$).
\item
$\gamma(\beta_s)=\beta_{-s}$.
\end{enumerate}
Hence we obtain the following corollary.

\begin{corollary}
\label{cor:orbimap2}
For any $r\in\QQQ$ and an integer $m \ge 3$,
the following hold.
\begin{enumerate}[\indent \rm (1)]
\item
The conjugacy classes of $\pi_1(\orbb(r;m))$
determined by the simple loops
$\beta_s$ and $\beta_{A_{(r;m)}(s)}$ are identical.
So, the same conclusion also holds for the
conjugacy classes of
the quotient group $\pi_1(\orbo(r;m))\cong\pi_1(\orbb(r;m))/\llangle \beta_{\infty}^2\rrangle$.

\item
The conjugacy classes of $\pi_1(\orbb(\infty;2))$
determined by the simple loops
$\beta_s$ and $\beta_{-s}$ are identical.
So, the same conclusion also holds for the conjugacy classes of
the quotient group $\pi_1(\orbo(r;m))\cong\pi_1(\orbb(\infty;2))/\llangle \beta_{r}^m\rrangle$.
\end{enumerate}
\end{corollary}

\begin{proof}[Proof of Theorem ~\ref{thm:fundametal_domain}]
Suppose that $s$ and $s'$ belong to the same $\RGPC{r}{n}$-orbit.
Since $\RGPC{r}{n}$ of automorphisms of $\DD$
is generated by the three transformations
$s\mapsto -s$,
$A_{(\infty;2)}$
and $A_{(r;m)}$ with $m=2n$,
we see by Corollary ~\ref{cor:orbimap2}
that the conjugacy classes $\beta_s$ and $\beta_{s'}$ in
$\pi_1(\orbo(r;m))$
are equal.
On the other hand,
we can easily see that
the natural action of $\pi_1(\OO)/\pi_1(\PConway)\cong (\ZZ/2\ZZ)^2$
on the conjugacy classes in $\pi_1(\PConway)$
preserves $\alpha_s$,
the pair of conjugacy classes represented by the loop $\alpha_s$
with two possible orientations.
So, the same conclusion holds for the
natural action of $\pi_1(\orbo(r:m))/\pi_1(\orbs(r;n))$
on the conjugacy classes in $\pi_1(\orbs(r;n))$.
Hence the precending result
implies that the conjugacy classes
$\alpha_s=\beta_s^2$ and $\alpha_{s'}=\beta_{s'}^2$ in
$\Hecke(r;n)=\pi_1(\orbs(r;n))$ are equal.
\end{proof}

\begin{proof}[Proof of Theorem ~\ref{thm:epimorophism}]
Suppose first that $s$ belongs to the
$\RGPC{r}{n}$-orbit of $\infty$.
Then the conjugacy class of $\alpha_s$ in
$\Hecke(r;n)\subset \pi_1(\orbo(r;m))$
is trivial by Theorem ~\ref{thm:fundametal_domain}.
Since the conjugacy class of
$\alpha_{\infty}=\beta_{\infty}^2$ in $\pi_1(\orbo(r;m))$
is also trivial by definition,
the homomorphism
$\pi_1(\PConway)\mapsto \pi_1(\orbo(r;m))$
descends to a homomorphism
\[
G(K(s))\cong \pi_1(\PConway)/ \llangle\alpha_{\infty},\alpha_s\rrangle
\to \pi_1(\OO)/ \llangle\beta_{\infty}^2,\beta_r^m\rrangle \cong
\pi_1(\orbo(r;m)).
\]
Since the image of this homomorphism is equal to the Heckoid group
$\Hecke(r;n)$
by Proposition ~\ref{prop:even_Heckoid_group} and
Definition ~\ref{def:odd_Heckoid_group},
we obtain an epimorphism $G(K(s)) \to \Hecke(r;n)$,
which is apparently upper-meridian-pair-preserving.

Suppose next that $s+1$ belongs to the $\RGPC{r}{n}$-orbit of $\infty$.
Then there is an epimorphism $G(K(s+1)) \to \Hecke(r;n)$ by the above argument.
Since there is an upper-meridian-pair-preserving
isomorphism $G(K(s))\cong G(K(s+1))$,
we obtain the desired epimorphism.
\end{proof}

At the end of this section,
we give a characterization of those rational numbers
which belong to the $\RGPC{r}{n}$-orbit of $\infty$.
Since $\Hecke(r;n)$ is isomorphic to $\Hecke(r+1;n)$,
we may assume in the remainder of this paper that $0<r<1$.
For the continued fraction expansion
\begin{center}
\begin{picture}(230,70)
\put(0,48){$\displaystyle{
r=[a_1,a_2, \dots,a_k]:=
\cfrac{1}{a_1+
\cfrac{1}{ \raisebox{-5pt}[0pt][0pt]{$a_2 \, + \, $}
\raisebox{-10pt}[0pt][0pt]{$\, \ddots \ $}
\raisebox{-12pt}[0pt][0pt]{$+ \, \cfrac{1}{a_k}$}
}}}$}
\end{picture}
\end{center}
where $k \ge 1$, $(a_1, \dots, a_k) \in (\mathbb{Z}_+)^k$, and $a_k \ge 2$,
let $\cfr$, $\cfr^{-1}$, $\epsilon\cfr$ and $\epsilon\cfr^{-1}$, with $\epsilon\in\{-,+\}$,
be the finite sequences defined as follows:
\begin{align*}
\cfr &= (a_1, a_2,\dots, a_k), \quad &
\cfr^{-1} &=(a_k,a_{k-1},\dots,a_1),\\
\epsilon\cfr &=(\epsilon a_1,\epsilon a_2,\dots,
\epsilon a_k), \quad &
\epsilon \cfr^{-1} &=(\epsilon a_k,\epsilon
a_{k-1},\dots,
\epsilon a_1).
\end{align*}
Then we have the following proposition,
which can be proved by the argument in \cite[Section ~5.1]{Ohtsuki-Riley-Sakuma}.

\begin{proposition}
\label{prop:continued fraction}
Let $r$ be as above and $n>1$ an integer or a half-integer.
Set $m=2n$.
Then a rational number $s$
belongs to the $\RGPC{r}{n}$-orbit of $\infty$
if and only if
$s$ has the following continued
fraction expansion{\rm :}
\[
s=
2c+[\epsilon_1\cfr,mc_1,-\epsilon_1\cfr^{-1},
2c_2,\epsilon_2\cfr,mc_3,-\epsilon_2\cfr^{-1},
\dots,
2c_{2t-2},\epsilon_t \cfr,mc_{2t-1},-\epsilon_t \cfr^{(-1)}]
\]
for some positive integer $t$, $c\in\ZZ$,
$(\epsilon_1,\epsilon_2,\dots,\epsilon_t) \in \{-,+\}^t$
and $(c_1,c_2,\dots,c_{2t-1})\in\ZZ^{2t-1}$.
\end{proposition}

\begin{remark}
\label{rem:Riley's_theorem}
{\rm
Riley's Theorem ~B and Theorem ~3 in \cite{Riley2}
imply the following.
Let $\alpha$ and $\beta$ be relatively prime integers with $1\le \beta <\alpha$.
For integers $d\ge 2$, $m\ge 3$, and $e\ge 1$,
consider the $2$-bridge link $K(\beta^*/\alpha^*)$,
where $(\alpha^*,\beta^*)=(\alpha^d m,\alpha^{d-1}m(\alpha-\beta)+e)$.
Then there is an epimorphism from the link group $G(K(\beta^*/\alpha^*))$ onto
the Heckoid group $\Hecke(\beta/\alpha;n)$, where $n=m/2$.
This result corresponds to the case when $r=(\alpha-\beta)/\alpha$
and $s=[\cfr,mc,-\cfr^{-1}]$, where $c=\epsilon\alpha^{d}$ with $\epsilon=\pm 1$
in Proposition ~\ref{prop:continued fraction}.
In fact, a simple calculation shows
\[
s=(\alpha^{d-1}m(\alpha-\beta)+(-1)^k\epsilon)/(\alpha^d m)=\beta^*/\alpha^*,
\]
where $k$ is the length of $\cfr$ and $\epsilon$ is chosen so that
$(-1)^k\epsilon=e$.
Thus
Theorem ~\ref{thm:epimorophism} and
Proposition ~\ref{prop:continued fraction} imply that
there is an epimorphism from
 the link group $G(K(\beta^*/\alpha^*))= G(K(s))$
onto the Heckoid group $\Hecke(r;n)\cong \Hecke(1-r;n)= \Hecke(\beta/\alpha;n)$,
recovering Riley's result.
}
\end{remark}

\section{Topological description of odd Heckoid orbifolds}
\label{sec:topological-description}

In this section,
we show, following the sketch of Agol ~\cite{Agol}, that
the orbifold $\orbo(r;m)$ and
the odd Heckoid orbifold $\orbs(r;n)$
are depicted as in Figures ~\ref{odd-Heckoid-orbifold1} and \ref{odd-Heckoid-orbifold2}.
Here, we employ the following convention.

\begin{figure}[htbp]
\begin{center}
\includegraphics{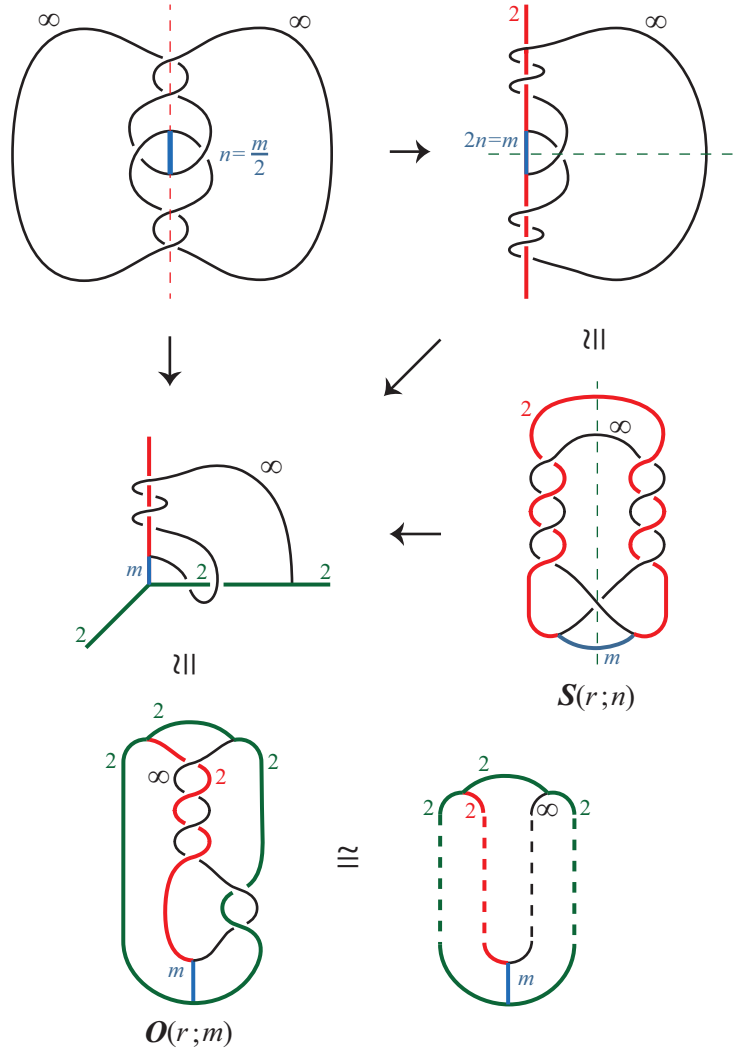}
\end{center}
\caption{\label{odd-Heckoid-orbifold1}
The case when $K(r)$ is a knot
and $m=2n>1$ is an odd integer.
Here $r=2/9=[4,2]$.
The odd Heckoid orbifold $\orbs(r;n)$ (middle right)
is a $\ZZ/2\ZZ$-covering of $\orbo(r;m)$ (lower left).
The upper left figure is not an orbifold, but
is a hyperbolic cone manifold.
The odd Heckoid orbifold $\orbs(r;n)$
is the quotient of the cone manifold
by the $\pi$-rotation around the axis containing the singular set.}
\end{figure}

\begin{figure}[htbp]
\begin{center}
\includegraphics{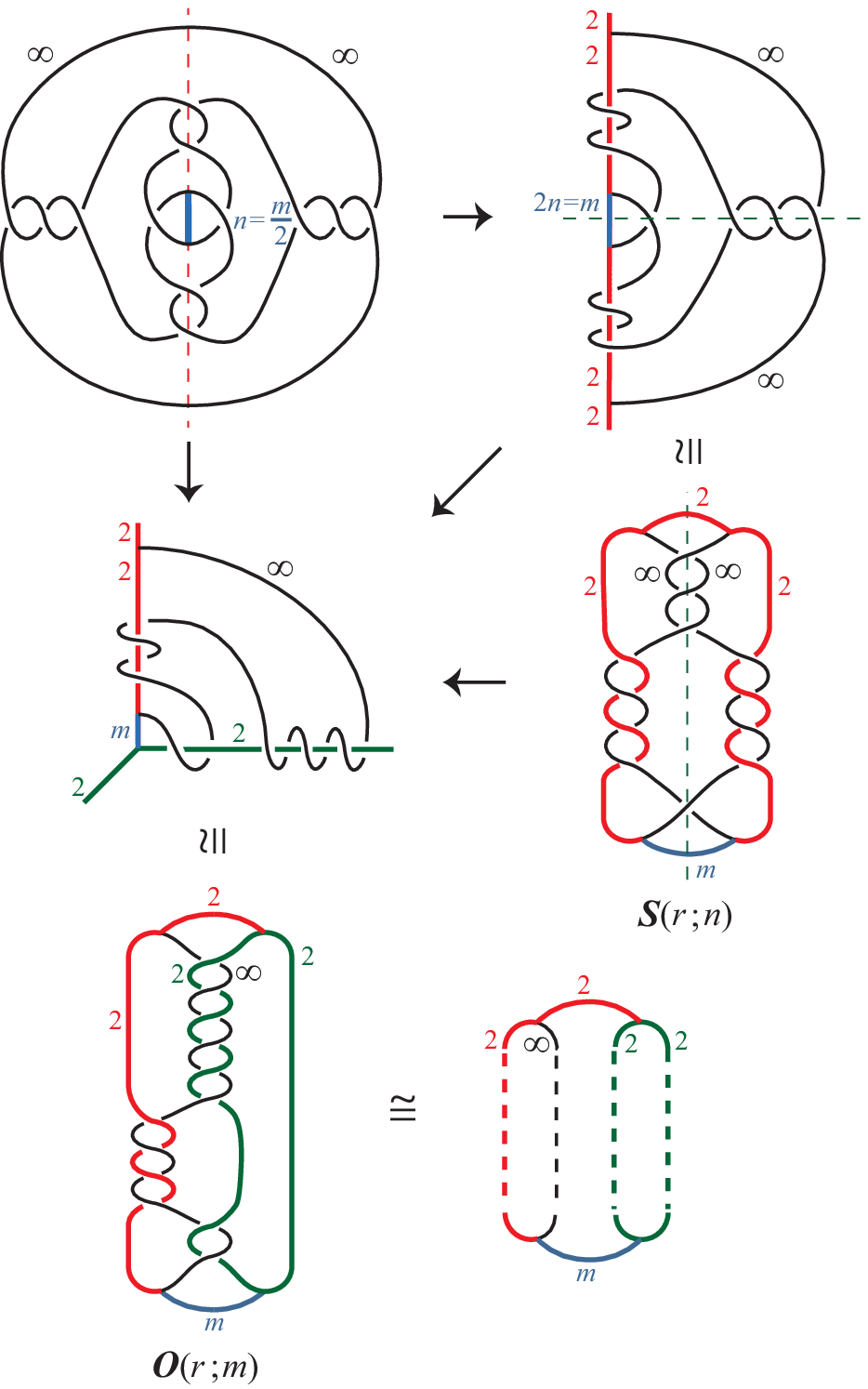}
\end{center}
\caption{\label{odd-Heckoid-orbifold2}
The case when $K(r)$ is a $2$-component link
and $m=2n>1$ is an odd integer.
Here $r=9/56=[6,4,2]$.
The odd Heckoid orbifold $\orbs(r;n)$ (middle right)
is a $\ZZ/2\ZZ$-covering of $\orbo(r;m)$ (lower left).
The upper left figure is not an orbifold, but
is a hyperbolic cone manifold.
The odd Heckoid orbifold $\orbs(r;n)$
is the quotient of the cone manifold
by the $\pi$-rotation around the axis containing the singular set.}
\end{figure}

\begin{convention}
\label{conv:orbifold}
{\rm
Let $\Sigma$ be a trivalent graph properly embedded in a
compact $3$-manifold $M$ such that each edge $e$ of $\Sigma$ is given
a {\it weight} $w(e)\in \NN_{\ge 2}\cup \{\infty\}$.
Here, a loop component of $\Sigma$ is regarded as an edge.
Assume that if $v$ is a (trivalent) vertex and $e_1,e_2,e_3$ are the edges incident on $v$,
then either some $w(e_i)$ is $\infty$ or the following inequality holds:
\[
\frac{1}{w(e_1)}+\frac{1}{w(e_2)}+\frac{1}{w(e_3)}>1.
\]
Then the weighted graph $(M,\Sigma,w)$ determines the following $3$-orbifold.
\begin{enumerate}[\indent \rm (a)]
\item
Let $\Sigma_{\infty}$ be the subgraph consisting of those edges with weight $\infty$.
Then the underlying space of the orbifold is the complement
of an open regular neighborhood of $\Sigma_{\infty}$.

\item
The singular set of the orbifold is the intersection of $\Sigma-\Sigma_{\infty}$
with the underlying space,
where the index is given by the weight.
(We identify an edge of the singular set with the corresponding edge of $\Sigma$.)
\end{enumerate}
We denote the orbifold by the same symbol $(M,\Sigma,w)$.
The part of the boundary of the orbifold $(M,\Sigma,w)$
contained in $\partial M$ is called the {\it outer-boundary} of $(M,\Sigma,w)$
and is denoted by $\partial_{out}(M,\Sigma,w)$.
}
\end{convention}

In this section, we prove the following propositions.

\begin{proposition}
\label{prop:quotient-Heckoid-orbifold}
For a rational number $r$ and an integer $m\ge 2$,
the orbifold $\orbo(r;m)$ is homeomorphic
to the orbifold $(S^3,K(r)\cup\tau_+\cup\tau_-,w)$,
where $\tau_+$ and $\tau_-$ are the upper and lower tunnels
of $K(r)$ and
the weight function $w$ is defined by the following rule.
\begin{enumerate}[\indent \rm (a)]
\item
$w(\tau_+)=2$ and $w(\tau_-)=m$.
\item
One of the four edges, say $J$, of $K(r)\cup\tau_+\cup\tau_-$
contained in $K(r)$ has weight $\infty$ and the
remaining three edges have weight $2$.
\end{enumerate}
\end{proposition}

\begin{proposition}
\label{prop:odd-Heckoid-orbifold}
For a rational number $r=q/p$,
where $p$ and $q$ are relatively prime integers such that $0\le q<p$,
and a non-integral half-integer
$n$ greater than $1$,
the odd Heckoid orbifold $\orbs(r;n)$
is described as follows.
\begin{enumerate}[\indent \rm (1)]
\item
Suppose that $K(r)$ is a knot, i.e., $p$ is odd
{\rm (}see Figure ~{\rm \ref{odd-Heckoid-orbifold1})}.
Consider the $2$-bridge knot $K(\hat r)$,
where $\hat r=(q/2)/p$ or $((p+q)/2)/p$
according to whether $q$ is even or odd.
Let $\tau_-$ be the lower tunnel of $K(\hat r)$, and
let $J_1$ and $J_2$ be the edges of $K(\hat r)\cup\tau_-$
such that $K(\hat r)=J_1\cup J_2$.
Then $\orbs(r;n)$ is homeomorphic to the orbifold
$(S^3,K(\hat r)\cup\tau_-,\hat w)$,
where the weight function $\hat w$ is defined as follows.
\begin{enumerate}[\indent \rm (a)]
\item
$\hat w(\tau_-)=m$ with $m=2n$.
\item
$\hat w(J_1)=\infty$ and $\hat w(J_2)=2$.
\end{enumerate}

\item
Suppose that $K(r)$ has two components, i.e., $p$ is even
{\rm (}see Figure ~{\rm \ref{odd-Heckoid-orbifold2})}.
Consider the $2$-bridge link $K(\hat r)$,
where $\hat r=q/(p/2)$.
Let $\tau_+$ and $\tau_-$ be the upper and lower tunnels of $K(\hat r)$,
and let $J_1$ and $J_2$ be the union of mutually disjoint arcs
of $K(\hat r)=t(\infty)\cup t(\hat r)$ bounded by $\partial(\tau_+\cup \tau_-)$
such that $K(\hat r)=J_1\cup J_2$
and such that $J_i\cap t(\infty)$ $(i=1,2)$ is
equal to the closure of the intersection of $t(\infty)$
with a component of $B^3-D_0$,
where $D_0$ is a ``horizontal'' disk embedded in $(B^3,t(\infty))$
bounded by the slope $0$ simple loop $\alpha_0$,
which intersects $t(\infty)$ transversely in two points and
contains the core tunnel $\tau_+$ of $(B^3,t(\infty))$
{\rm (}see Figure ~{\rm \ref{covering-rational-tangle}(b))}.
Then $\orbs(r;n)$ is homeomorphic to the orbifold
$(S^3,K(\hat r)\cup\tau_+\cup\tau_-,\hat w)$,
where the weight function $\hat w$ is defined as follows.
\begin{enumerate}[\indent \rm (a)]
\item
$\hat w(\tau_+)=2$ and $\hat w(\tau_-)=m$.
\item
The {\rm (}two{\rm )} components of $J_1$ have weight $\infty$,
and the {\rm (}two{\rm )} components of $J_2$ have weight $2$.
\end{enumerate}
\end{enumerate}
\end{proposition}

\begin{remark}
{\rm
(1) Because of the $(\ZZ/2\ZZ)^2$-symmetry of $2$-bridge links,
the choice of the edge $J$ in $K(r)$
in Proposition ~\ref{prop:quotient-Heckoid-orbifold}
and that of the edges $J_1$ and $J_2$ in $K(\hat r)$
in Proposition ~\ref{prop:odd-Heckoid-orbifold}
do not affect the homeomorphism class of the resulting orbifolds.

(2) By the announcement in \cite[Section ~3 of Preface]{ASWY},
there exist hyperbolic cone manifolds as illustrated in the upper left figures
in Figures ~\ref{odd-Heckoid-orbifold1} and \ref{odd-Heckoid-orbifold2}.
The odd Heckoid orbifolds are $\ZZ/2\ZZ$-quotients of the cone manifolds.}
\end{remark}

\begin{proof}[Proof of Proposition ~\ref{prop:quotient-Heckoid-orbifold}]
Recall that
$\orbo(r;m)=\orbb(\infty;2)\cup \orbb(r;m)$
and note that
$(S^3,K(r)\cup\tau_+\cup\tau_-,w)=
(B^3,t(\infty)\cup\tau_+,w_+)\cup (B^3,t(r)\cup\tau_-,w_-)$,
where $w_{\pm}$ are \lq\lq restrictions'' of $w$.
We can observe that there are homeomorphisms
$f_+:\orbb(\infty;2)\to (B^3,t(\infty)\cup\tau_+,w_+)$ and
$f_-:\orbb(r;m)\to (B^3,t(r)\cup\tau_-,w_-)$
such that the restriction of each of $f_{\pm}$
to the outer-boundary determines a homeomorphism
from $\check\OO$ to the $2$-orbifold, $\check\PConway$,
obtained from the Conway sphere $\PConway$
by removing an open regular neighborhood of a puncture
and filling in order $2$ cone points to the remaining punctures.
Moreover, each of the homeomorphisms
maps the (isotopy class of the) simple loop $\beta_s$ in $\OO$
to the the (isotopy class of the) simple loop $\alpha_s$ in $\PConway$
for every $s\in\QQQ$.
Thus we can choose $f_{\pm}$ so that they are consistent with the gluing
maps in the constructions of $\orbo(r;m)$ and $(S^3,K(r)\cup\tau_+\cup\tau_-,w)$;
so, $f_+$ and $f_-$ determine the desired homeomorphism
from $\orbo(r;m)$ to $(S^3,K(r)\cup\tau_+\cup\tau_-,w)$.
\end{proof}

\begin{proof}[Proof of Proposition ~\ref{prop:odd-Heckoid-orbifold}]
By Definition ~\ref{def:odd_Heckoid_group},
the odd Heckoid group $\Hecke(r;n)$ is equal to the kernel
of the natural projection
\[
\pi_1(\orbo(r;m))\to \pi_1(\orbo(r;m))/\psi(\pi_1(\PConway)),
\]
where $m=2n$ is an odd integer.
Thus the Heckoid orbifold $\orbs(r;n)$
is the regular covering of $\orbo(r;m)$
with the covering transformation group
\[
\pi_1(\orbo(r;m))/\psi(\pi_1(\PConway))
\cong
\pi_1(\OO)/\llangle \pi_1(\PConway), \beta_{\infty}^2, \beta_r^m\rrangle.
\]
Note that $\pi_1(\OO)/\pi_1(\PConway)\cong (\ZZ/2\ZZ)^2$
is generated by the homology classes $[\beta_0]$ and $[\beta_{\infty}]$.
Since $[\beta_{r}]=p[\beta_0]+q[\beta_{\infty}]$ and since $m$ is odd,
the covering transformation group is isomorphic to
\[
\langle [\beta_0], [\beta_{\infty}] \ | \ p[\beta_0]+q[\beta_{\infty}]\rangle_{(2)}
\cong \ZZ/2\ZZ,
\]
where the suffix $(2)$ represents that this is a presentation
as a $\ZZ/2\ZZ$-module.
Let $\tilde\orbb(\infty;2)$ and $\tilde\orbb(r;m)$, respectively,
be the inverse images of the suborbifolds $\orbb(\infty;2)$ and $\orbb(r;m)$
under the $2$-fold covering $\orbs(r;n)\to \orbo(r;m)$.
Then, by the above description of the covering transformation group,
the covering orbifold $\tilde\orbb(\infty;2)$ and its covering involution, $h_+$,
are described as follows.
\begin{enumerate}[\indent \rm (a)]
\item
If $(p,q)\equiv (1,0) \pmod2$,
then $\tilde\orbb(\infty;2)$ is identified with the orbifold
$(B^3,t(\infty),\hat w_+)$, where the weight function $\hat w_+$
takes the value $\infty$ on one of the components of $t(\infty)$
and the value $2$ on the other component.
Under this identification,
the covering involution $h_+$ is the $\pi$-rotation
whose axis contains the core tunnel
(see Figure ~\ref{covering-rational-tangle}(a)).
\item
If $(p,q)\equiv (0,1) \pmod2$,
then $\tilde\orbb(\infty;2)$ is identified with the orbifold
$(B^3,t(\infty)\cup\tau_+,\hat w_+)$,
where $\tau_+$ is the core tunnel, and the weight function $\hat w_+$
is given by the following rule:
$\hat w_+(\tau_+)=2$, and $\hat w_+$
takes the value $\infty$
on a pair of edges whose interiors are
contained in one of the components the complement
of the horizontal disk $D_0$ in $B^3$,
and the value $2$ on
the remaining pair of edges.
Under this identification,
$h_+$ is the $\pi$-rotation
whose axis bisects $\tau_+$ (see Figure ~\ref{covering-rational-tangle}(b)).
\end{enumerate}

\begin{figure}[htbp]
\begin{center}
\includegraphics{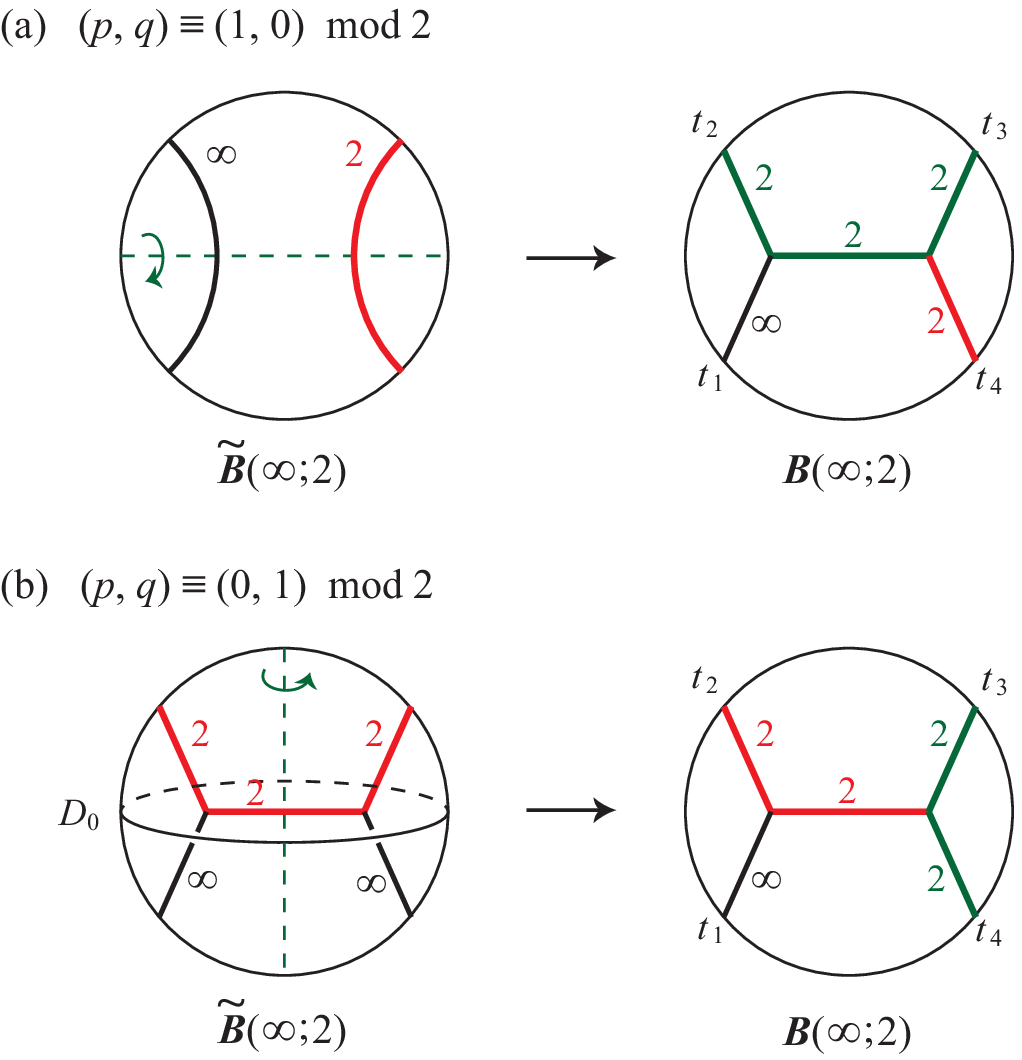}
\end{center}
\caption{\label{covering-rational-tangle}
The covering orbifold $\tilde\orbb(\infty;2)$ of $\orbb(\infty;2)$}
\end{figure}

We can easily observe the following:
\medskip

{\bf Claim.}
{\it
Under the identifications of the outer-boundaries
$\partial_{out} \tilde\orbb(\infty;2)$ and $\partial_{out} \orbb(\infty;2)$
with {\rm (}an orbifold obtained from{\rm )} $\PConway$,
as in the above and in the proof of
Proposition ~\ref{prop:quotient-Heckoid-orbifold},
the covering projection
$\partial_{out} \tilde\orbb(\infty;2) \to
\partial_{out} \orbb(\infty;2)$
maps the pair of simple loops $(\alpha_0,\alpha_{\infty})$
to $(\alpha_0,\alpha_{\infty}^2)$ or $(\alpha_0^2,\alpha_{\infty})$
according to whether $(p,q)\equiv (1,0)$ or $(0,1) \pmod2$.
}
\medskip

On the other hand,
$\tilde\orbb(r;m)$ is identified with the orbifold
$(B^3,t_-\cup\tau_-,\hat w_-)$,
where $(B^3,t_-)$ is a $2$-string trivial tangle,
$\tau_-$ is the core tunnel of $(B^3,t_-)$,
and where the weight function $\hat w_-$ is given by the following rule:
$\hat w_-(\tau_-)=m$, and $\hat w_-$ takes the value $\infty$
on one of the four edges of $t_-\cup\tau_-$
whose union is equal to $t_-$,
and the value $2$ on the remaining three edges.
The covering involution, $h_-$, of
$\tilde\orbb(r;m)\cong(B^3,t_-\cup\tau_-,\hat w_-)$
is the $\pi$-rotation
whose axis bisects $\tau_-$ (cf. Figure ~\ref{covering-rational-tangle}(b)).

By these observations concerning the suborbifolds
$\tilde\orbb(\infty;2)$ and $\tilde\orbb(r;m)$,
the odd Heckoid orbifold $\orbs(r;n)=\tilde\orbb(\infty;2)\cup\tilde\orbb(r;m)$
is regarded as the union of the orbifold $(B^3,t_-\cup\tau_-,\hat w_-)$
and the orbifold
$(B^3,t(\infty),\hat w_+)$ or $(B^3,t(\infty)\cup\tau_+,\hat w_+)$
according to whether $(p,q)\equiv (1,0)$ or $(0,1) \pmod2$.
This implies that
$\orbo(r;n)$ is constructed from some $2$-bridge link
as in the proposition.
The remaining task is to identify the slope, $\hat r$,
of the $2$-bridge link.
To this end, pick a disk $\tilde D$
properly embedded in $(B^3,t_-\cup\tau_-,\hat w_-)\cong \tilde\orbb(r;m)$
which intersects the singular set transversely
in a single point in the interior of $\tau_-$,
such that $\tilde D$ is mapped homeomorphically by the covering projection
to a disk in $\orbb(r;m)$
bounded by the loop $\alpha_r$.
Then the slope $\hat r$ of the $2$-bridge link
is equal to the slope of the simple loop $\partial\tilde D$ in
$\partial_{out} \tilde\orbb(\infty;2)$.
(Here
$\partial_{out} \tilde\orbb(\infty;2)$ is identified with
the outer boundary of
$(B^3,t(\infty),\hat w_+)$ or $(B^3,t(\infty)\cup\tau_+,\hat w_+)$;
so the slope of $\partial\tilde D$ in it is defined.)
By using the Claim in the above,
we can see that $\hat r=(q/2)/p$ or $q/(p/2)$
according as $(p,q)\equiv (1,0)$ or $(0,1) \pmod2$.
This completes the proof of the proposition
except when $(p,q)\equiv (1,1)\pmod2$.
This remaining case can be settled by using the fact that
there is a homeomorphism
from $(S^3, K(q/p))$ to $(S^3, K((p+q)/p))$
sending the upper/lower tunnels of $K(q/p)$
to those of $K((p+q)/p)$.
\end{proof}

\section{Heckoid groups as two-parabolic Kleinian groups}
\label{sec:Kleinian-Heckod-groups}

In this section, we prove Theorem ~\ref{thm.Kleinian_heckoid},
which is contained in the announcement by Agol ~\cite{Agol}.
As noted in \cite{Agol},
the proof relies on the orbifold theorem
and is analogous to the arguments in
\cite[Proof of Theorem ~9]{Jones-Reid}.

\begin{remark}
\label{rem:ASWY}
{\rm
This theorem also follows from the announcement
made in the second author's joint work with
Akiyoshi, Wada and Yamashita ~\cite[Section ~3 of Preface]{ASWY}.
Note, however, that there is an error
in the assertion 5 in Page IX in Preface,
though a special case is treated correctly in
\cite[Proposition ~5.3.9]{ASWY}.
In fact, the first sentence of the assertion
should be read as follows:
The holonomy group of $M(\theta^-,\theta^+)$ is discrete if and only if
$\theta^{\pm}\in \{2\pi/n \svert n\in \frac{1}{2}\NN_{\ge 2}\}\cup\{0\}$.
The second author would also like to note that this assertion can be proved
by using the argument of Parkkonen in \cite[Lemma ~7.5]{Parkkonen};
this was forgotten to mention in \cite{ASWY},
though the paper is included in the bibliography.
}
\end{remark}

In order to prove Theorem ~\ref{thm.Kleinian_heckoid},
we prove that $\orbo(r;m)$ with $m=2n\ge 3$
admits a hyperbolic structure.
Throughout this section,
we identify $\orbo(r;m)$ with the orbifold $(S^3,K(r)\cup\tau_+\cup\tau_-,w)$
in Proposition ~\ref{prop:quotient-Heckoid-orbifold}.
We denote by $B^3_+$ and $B^3_-$
the $3$-balls of $S^3$ bounded by the bridge sphere of $K(r)$ such that
\[
(S^3,K(r)\cup\tau_+\cup\tau_-,w)=
(B^3_+,t(\infty)\cup\tau_+,w_+)\cup (B^3_-,t(r)\cup\tau_-,w_-).
\]
We refer to \cite[Introduction and Section ~8]{Boileau-Porti}
(cf. \cite[Chapter ~2]{Boileau-Maillot-Porti},
\cite[Chapter ~2]{CHK}
and \cite[Chapter ~6]{Kapovich})
for standard terminologies for orbifolds.

\begin{lemma}
\label{lem:irreducible}
For a rational number $r$ and an integer $m\ge 3$,
the following hold.
\begin{enumerate}[\indent \rm (1)]
\item
$\orbo(r;m)$ does not contain a bad $2$-suborbifold.
\item
Any football $S^2(p,p)$ in $\orbo(r;m)$ bounds a discal $3$-suborbifold.
\item
$\orbo(r;m)$ does not contain an essential turnover.
\item
$\orbo(r;m)$ is topologically atoroidal, i.e., it
does not contain an essential orientable toric $2$-suborbifold.
\end{enumerate}
\end{lemma}

\begin{proof}
{\rm
(1) Suppose that
$\orbo(r;m)$ contains a bad $2$-suborbifold, $F$.
Then $F$ is either
a {\it teardrop} $S^2(p)$
or a {\it spindle} $S^2(p,q)$ with $1<p<q$.
Since the indices of the singular set of $\orbo(r;m)$ are $2$ and $m(\ge 3)$,
and since the underlying $2$-sphere $|F|$ intersects $K(r)$
in an even number of points,
we see that
$|F|$ is disjoint from $K(r)$ and intersects (at least)
one of the unknotting tunnels $\tau_{\pm}$ transversely in a single point,
where $F\cong S^2(2)$, $S^2(m)$ or $S^2(2,m)$.
Since the endpoints of each of the unknotting tunnels are contained in $K(r)$,
this implies that $K(r)$ is a split link, a contradiction.
Hence $\orbo(r;m)$ cannot contain a bad $2$-suborbifold.

(2) Let $F$ be a suborbifold of $\orbo(r;m)$ which is a football.
As in (1), we see that one of the following holds.
\begin{enumerate}[\indent \rm (i)]
\item
$|F|$ intersects $K(r)$ in two points, where $F\cong S^2(2,2)$.
\item
$|F|$ is disjoint from $K(r)$ and intersects
one of the unknotting tunnels $\tau_{\pm}$ in two points
and does not intersect the other unknotting tunnel,
where $F\cong S^2(2,2)$ or $S^2(m,m)$.
\end{enumerate}
Suppose that condition (i) holds.
Then $|F|$ is disjoint from $\tau_+\cup\tau_-$,
and so either $\tau_+$ and $\tau_-$ are separated by $|F|$,
or $\tau_+\cup\tau_-$ is contained in a single component of $S^3-|F|$.
If $\tau_+$ and $\tau_-$ are separated by $|F|$,
then $|F|$ must intersect $K(r)$ in at least four points, a contradiction.
Hence $\tau_+\cup\tau_-$ is contained in a single component of $S^3-|F|$.
Let $B^3_1$ and $B^3_2$ be the $3$-balls in $S^3$ bounded
by $|F|$, such that $\tau_+\cup\tau_-\subset B^3_2$.
Set $K_i=B^3_i\cap K(r)$ ($i=1,2$).
Then the genus $3$ open handle body
$S^3-(K(r)\cup\tau^+\cup\tau^-)$
is the union of $B^3_1-K_1$ and $B^3_2-(K_2\cup\tau^+\cup\tau^-)$
along the open annuls $|F|-K(r)$,
and hence the rank $3$ free group,
$\pi_1(S^3-(K(r)\cup\tau^+\cup\tau^-))$, is the free product
of $\pi_1(B^3_1-K_1)$ and $\pi_1(B^3_2-(K_2\cup\tau^+\cup\tau^-))$
with the infinite cyclic amalgamated subgroup $\pi_1(|F|-K(r))$.
Since $H_1(B^3_1-K_1)\cong\ZZ$,
this implies $\pi_1(B^3_1-K_1)\cong \ZZ$.
Hence $(B^3_1,K_1)$ is a trivial $1$-string tangle.
Thus
$(B^3_1,B^3_1\cap (K(r)\cup\tau_+\cup\tau_-))=(B^3_F,B^3_1\cap K(r))$
determines a discal $3$-suborbifold of $\orbo(r;m)$ bounded by $F$,
and therefore $F$ is inessential.

Suppose that condition (ii) holds.
For simplicity, we assume that $|F|$ intersects $\tau_+$ in two points
and does not intersect $\tau_-$.
(The other case is treated similarly.)
Let $B^3_F$ be the $3$-ball bounded by $|F|$ such that
$B^3_F\cap\tau_+$ is a subarc of $\tau_+$.
Then $(B^3_F,B^3_F\cap (K(r)\cup\tau_+\cup\tau_-))=(B^3_F,B^3_F\cap \tau_+)$
is a trivial $1$-string tangle,
because $\tau_+$ is contained in
a trivial constituent knot in the spatial graph $K(r)\cup \tau_+\cup\tau_-$.
Hence it determines a discal $3$-suborbifold of $\orbo(r;m)$ bounded by $F$.

(3) Suppose that $\orbo(r;m)$ contains an essential turnover
$F\cong S^2(p,q,r)$.
Then either $|F|$ is disjoint from $K(r)$, or
$|F|$ intersects $K(r)$ in two points.
In the first case, $|F|$ intersects $\tau_+\cup \tau_-$ in three points
and hence $|F|$ intersects $\tau_+$ or $\tau_-$
in an odd number of points.
As in (1), it follows that $K(r)$ is a split link, a contradiction.
Hence we may assume that $|F|$ intersects $K(r)$ in two points,
and therefore $|F|$ is disjoint from $\tau_+$ or $\tau_-$.
For simplicity, we assume that $|F|$ is disjoint from $\tau_-$.
(The other case is treated similarly.)
By using the fact that
$|F|$ is also disjoint from $\partial \orbo(r;m)$
and the fact that
$(B^3_-,t(r)\cup \tau_-)$ is a relative regular neighborhood of
$\tau_-$ in $(S^3,K(r)\cup\tau_+\cup\tau_-)$,
we can see that $F$ is isotopic to a $2$-suborbifold
which is disjoint from the suborbifold $(B^3_-,t(r)\cup \tau_-,w_-)$.
Hence we may assume that
$F$ is contained in the interior of
the suborbifold $(B^3_+, t(\infty)\cup\tau_+,w_+)$.
Let $t_i$ ($1\le i\le 4$) be the edges of $t(\infty)\cup\tau_+$
as illustrated in the right figures in Figure ~\ref{covering-rational-tangle}.
Note that $t(\infty)=\cup_{i=1}^4 t_i$, $w_+(t_1)=\infty$ and $w_+(t_i)=2$ ($2\le i\le 4$).
Thus $|F|$ is disjoint from $t_1$ and $|F|$ intersects $(t(\infty)-t_1)\cup \tau_+$
transversely in three points.
Let $D_h$ be the disk properly embedded in $B^3_+$
determined by the plane in which Figure ~\ref{covering-rational-tangle} is drawn.
Then $D_h$ contains the graph $t(\infty)\cup \tau_+$.
We may assume that $|F|$ is transversal to $D_h$ and hence
$|F|\cap D_h$ consists of mutually disjoint circles.
By using the irreducibility of $B^3_+-(t(\infty)\cup \tau_+)$,
we may assume, by a standard argument, that
no component of $|F|\cap D_h$ bounds a disk
disjoint from $t(\infty)\cup \tau_+$.
Then it follows that
$|F|\cap D_h$ must consist of a single circle
which intersects $\tau_+$, $t_3$ and $t_4$ in a single point.
Let $D_F$ be the disk in $D_h$ bounded by the circle $D_h\cap |F|$,
and let $B_F^3$ be the $3$-ball in $B_+^3$ bounded by $|F|$.
Then $D_F$ is properly embedded in $B_F^3$,
and $B_F^3\cap (t(\infty)\cup\tau_+)=D_F\cap (t(\infty)\cup\tau_+)$.
Hence $(B_F^3,B_F^3\cap (t(\infty)\cup\tau_+))$
determines a discal $3$-orbifold
bounded by the turnover $F$,
a contradiction.

(4)
Suppose that $\orbo(r;m)$ contains an essential pillow $F\cong S^2(2,2,2,2)$.
Then $|F|$ is disjoint from $\tau_-$,
which has index $m\ge 3$, and hence
we may assume, as in (3), that $|F|$ is contained in
the suborbifold $(B^3_+, t(\infty)\cup\tau_+,w_+)$.
Under the notation in (3),
$|F|$ is disjoint from $t_1$ and $|F|$ intersects $t(\infty)\cup \tau_+$
transversely in four points.
We may also assume that $|F|$ is transversal to the disk $D_h$
introduced in (3) and hence
$|F|\cap D_h$ consists of mutually disjoint circles.
By using the irreducibility of $B^3_+-(t(\infty)\cup \tau_+)$
and the assumption that $F$ is essential,
we may assume that
no component of $|F|\cap D_h$ bounds a disk disjoint from
$t(\infty)\cup \tau_+$.
Hence we see that
either
(i) $|F|\cap D_h$ consists of a single loop
which intersects $(t(\infty)-t_1)\cup\tau_+$ in four points,
or (ii)
$|F|\cap D_h$ consists of two loops
each of which intersects $(t(\infty)-t_1)\cup\tau_+$ in two points.
In either case, we can find an ``outermost disk'' $\delta$ in $D_h$
satisfying the following conditions.
\begin{enumerate}[\indent \rm (a)]
\item
$\delta\cap |F|$ is an arc, $c$, in $\partial \delta$.
\item
$\delta\cap(t(\infty)\cup\tau_+)$
is an arc, $c'$, in $\partial \delta$ which is contained in the interior
of an edge of the graph $t(\infty)\cup\tau_+$.
\item
$\partial\delta=c\cup c'$.
\item
$\delta$ is contained in the $3$-ball, $B_F^3$, in $B^3_+$
which is bounded by $|F|$.
\end{enumerate}
Then the frontier of a regular neighborhood of
$\delta$ in $B_F^3$ is a disk properly embedded in $B_F^3$
disjoint from the singular set,
whose boundary is an essential loop in the pillow $F$.
This contradicts the assumption that $F$ is essential.
Hence $\orbo(r;m)$ does not contain an essential pillow.

Assume that $\orbo(r;m)$ contains an essential torus, $F$.
Then $F$ is a torus contained in $S^3-(K(r)\cup\tau_+\cup\tau_-)$,
which is a genus $3$ open handlebody.
Hence $F$ must be compressible in $S^3-(K(r)\cup\tau_+\cup\tau_-)$,
a contradiction.

By the classification of toric $2$-orbifolds,
an orientable toric $2$-orbifold is a torus, a turnover or a pillow.
Hence by (3) and the above arguments,
$\orbo(r;m)$ does not contain an essential orientable toric $2$-orbifold.
}
\end{proof}

\begin{lemma}
\label{lem:Haken}
For a rational number $r$ and an integer $m\ge 3$,
the orbifold $\orbo(r;m)$ is Haken,
i.e., it is irreducible and does not contain an essential turnover,
but contains an essential $2$-suborbifold.
\end{lemma}

\begin{proof}
{\rm
By Lemma ~\ref{lem:irreducible}(1), $\orbo(r;m)$
does not contain a bad $2$-suborbifold.
By Lemma ~\ref{lem:irreducible}(2) and (3),
every orientable spherical $2$-suborbifold of $\orbo(r;m)$
with nonempty singular set
bounds a discal $3$-suborbifold.
Moreover any $2$-sphere (i.e., spherical $2$-suborbifold
with empty singular set) of $\orbo(r;m)$
bounds a $3$-ball in $\orbo(r;m)$,
because Proposition ~\ref{prop:quotient-Heckoid-orbifold} implies that
the complement of an open regular neighborhood of the singular set of $\orbo(r;m)$
is homeomorphic to a genus $3$ handlebody.
Hence the orbifold $\orbo(r;m)$ is irreducible.
Moreover, it does not contain an essential turnover
by Lemma ~\ref{lem:irreducible}(3).
Since $\partial \orbo(r;m)\cong S^2(2,2,2,m)$ is not a turnover,
we see by \cite[Proposition ~4.6]{Boileau-Maillot-Porti}
that $\orbo(r;m)$ is Haken
(see \cite[Definition ~8.0.1]{Boileau-Porti}).
}
\end{proof}

\begin{lemma}
\label{lem:hyperbolic}
For a rational number $r$ and an integer $m\ge 3$,
the orbifold $\orbo(r;m)$ is homotopically atoroidal, i.e.,
$\pi_1(\orbo(r;m))$ is not virtually abelian and
every rank $2$ free abelian subgroup of $\pi_1(\orbo(r;m))$ is peripheral.
\end{lemma}

\begin{proof}
{\rm
Since $\orbo(r;m)$ is Haken by Lemma ~\ref{lem:Haken},
we see by \cite[Theorem ~A]{Takeuchi}
(cf. \cite[Proposition ~8.2.2]{Boileau-Porti})
that $\orbo(r;m)$ is good,
i.e., it has a manifold cover.
Suppose on the contrary that
$\orbo(r;m)$ is not homotopically atoroidal
(see \cite[Definition ~8.2.13]{Boileau-Porti}).
Then, since $\orbo(r;m)$ is topologically atoroidal
by Lemma ~\ref{lem:irreducible}(4),
we see by
\cite[Proposition ~8.2.11]{Boileau-Porti} that
$\orbo(r;m)$ is either Euclidean or Seifert fibered.
(Here, we use the fact that $\orbo(r;m)$ is good.)
This contradicts the fact that
$\partial \orbo(r;m)\cong S^2(2,2,2,m)$ is not Euclidian.
Hence the orbifold $\orbo(r;m)$ is homotopically atoroidal.
}
\end{proof}

\begin{corollary}
\label{cor:hyperbolic}
For a rational number $r$ and an integer $m\ge 3$,
the interior of $\orbo(r;m)$ has a geometrically finite hyperbolic structure.
In particular, $\orbo(r;m)$ is very good,
i.e., it has a finite cover which is a manifold.
\end{corollary}

\begin{proof}
By Lemmas \ref{lem:Haken} and \ref{lem:hyperbolic},
$\orbo(r;m)$ is a homotopically atoroidal Haken $3$-orbifold.
Hence, by the orbifold theorem for Haken orbifolds
\cite[Theorem ~8.2.14]{Boileau-Porti},
$\orbo(r;m)$ is hyperbolic.
Moreover, it follows from the proof of the theorem that the hyperbolic structure can
be chosen to be geometrically finite.
The last assertion follows from Selberg's Lemma
\cite{Selberg}
(cf. \cite[Theorem ~2.29]{Matsuzaki-Taniguchi}).
\end{proof}

Let $P=\mathrm{cl}(\partial \orbb(\infty;2)-\partial_{out} \orbb(\infty;2))$.
Then $P\cong D^2(2,2)$ is an annular $2$-suborbifold
in $\partial \orbo(r;m)$,
and the following lemma shows that
$(\orbo(r;m),P)$ is a pared $3$-orbifold
(see \cite[Definition ~8.3.7]{Boileau-Porti}).

\begin{lemma}
\label{lem:pared_orbifold}
For a rational number $r$ and an integer $m\ge 3$,
the pair $(\orbo(r;m),P)$ satisfies the following conditions,
and hence it is a pared $3$-orbifold.
\begin{enumerate}[\indent \rm (1)]
\item
$\orbo(r;m)$ is irreducible and very good.
\item
$P$ is incompressible.
\item
Every rank $2$ free abelian subgroup of $\pi_1(\orbo(r;m))$ is
conjugate to a subgroup of $\pi_1(P)$.
(In fact, $\pi_1(\orbo(r;m))$ does not contain a rank $2$ free abelian subgroup.)
\item
Any properly embedded annular $2$-suborbifold
$(A,\partial A)\subset (\orbo(r;m),P)$
whose boundary rests on essential loops in $P$
is parallel to $P$.
\end{enumerate}
\end{lemma}

\begin{proof}
{\rm
(1) This follows from Lemma ~\ref{lem:Haken} and Corollary ~\ref{cor:hyperbolic}.

(2) Suppose that $P$ is compressible.
Then there is a discal orbifold $(F,\partial F)$ properly embedded in
$(\orbo(r;m),P)$ such that $\partial F$ is a loop in $P$ parallel to $\partial P$.
Since $F$ has at most one cone point,
$|F|$ is disjoint from $\tau_+$ or $\tau_-$.
For simplicity, we assume that $|F|$ is disjoint from $\tau_-$.
(The other case is treated similarly.)
By using the fact that
$\partial F$ is parallel to
$\partial P$ in $\partial\orbo(r;n)$,
and the fact that
$(B^3_-,t(r)\cup \tau_-)$ is a relative regular neighborhood of
$\tau_-$ in $(S^3,K(r)\cup\tau_+\cup\tau_-)$,
we can see that $F$ is isotopic to a $2$-suborbifold
which is disjoint from the suborbifold $(B^3_-,t(r)\cup \tau_-,w_-)$.
Hence we may assume that
$F$ is contained in the interior of
the suborbifold $(B^3_+, t(\infty)\cup\tau_+,w_+)$.
Then, by looking at the intersection of $|F|$ with the disk $D_h$
as in the proof of Lemma ~\ref{lem:irreducible}(3),
we see that this cannot happen.

(3) Suppose that $\pi_1(\orbo(r;m))$ contains
a rank $2$ free abelian subgroup, $H$.
Then $H$ is conjugate to a subgroup of $j_*(\pi_1(\partial\orbo(r;m)))$
by Lemma ~\ref{lem:hyperbolic},
where $j$ is the inclusion.
If $j_*$ is injective, then $j_*(\pi_1(\partial\orbo(r;m)))\cong \pi_1(S^2(2,2,2,m))$
is isomorphic to a Fuchsian group and
hence it cannot contain a rank $2$ free abelian subgroup, a contradiction.
So, $j_*$ is not injective.
By the loop theorem for good orbifolds \cite[p.133]{Boileau-Porti},
$\partial \orbo(r;m)$ is compressible.
Let $F\cong D^2(d)$ be a compressing disk for $\partial \orbo(r;m)$.
Then $d=1$, $2$ or $m$, and
$\partial F$ is a loop in $\partial \orbo(r;m)$
separating the $4$ singular points into two pairs of singular points.
Thus the result of compression of $\partial \orbo(r;m)$ by $F$
is a union of two $2$-suborbifolds, $F_1\cong S^2(2,2,d)$ and $F_2\cong S^2(2,m,d)$.
If $d=1$, $F_2\cong S^2(2,m)$ is a bad $2$-suborbifold,
a contradiction to Lemma ~\ref{lem:irreducible}(1).
Hence $d=2$ or $m$, and therefore $F_1$ and $F_2$ are turnovers.
By Lemma ~\ref{lem:irreducible}(3), they must be inessential.
Since none of them is boundary parallel,
each $F_i$ is a spherical turnover
bounding a discal $3$-orbifold.
Note that the singular set of $\orbo(r;m)$ has exactly two vertices
and the boundaries of regular neighborhoods of the vertices are
$S^2(2,2,2)$ and $S^2(2,2,m)$ (see Proposition ~\ref{prop:quotient-Heckoid-orbifold}).
Hence we see $d=2$ and
$F_1$ and $F_2$ are the boundaries of regular neighborhoods of the two vertices.
Thus $\partial \orbo(r;m)$ is parallel to the boundary of the $3$-orbifold
obtained from the regular neighborhoods of the two vertices of the singular set
by joining them by a tube around the unique edge of the singular set (of index $2$)
joining the two vertices.
Thus $\pi_1(\orbo(r;m))$ is a free product
of the dihedral groups of orders $4$ and $2m$
with amalgamated subgroup isomorphic to $\ZZ/2\ZZ$.
It is easy to see that such a group cannot contain
a rank $2$ free abelian subgroup.
Hence,  $\pi_1(\orbo(r;m))$ does not contain
a rank $2$ free abelian subgroup.

(4) Let $(A,\partial A)$ be an annular $2$-suborbifold
properly embedded in $(\orbo(r;m),P)$
whose boundary rests on essential loops in $P$.

Suppose first that $A$ is an annulus.
Then $A$ is disjoint from $\tau_+\cup \tau_-$.
Since each component of $\partial A$ is parallel to $\partial P$
in $P\subset \partial\orbo(r;m)$,
we may assume as in (2) that
$A$ is embedded in a regular neighborhood
of the $2$-suborbifold of
$(S^3,K(r)\cup\tau_+\cup \tau_-,w)$
determined by the $2$-bridge sphere.
Thus $(A,\partial A)$
is regarded as a suborbifold of
$(\check\PConway\times [-1,1], P')$,
where $\check\PConway$ is
obtained from the Conway sphere $\PConway$
by removing an open regular neighborhood of a puncture
and filling in order $2$ cone points to the remaining punctures
(cf. Proof of Proposition ~\ref{prop:quotient-Heckoid-orbifold}),
and $P'$ is the product annulus $\partial \check\PConway\times [-1,1]$.
Consider a disk properly embedded in
$|\check\PConway|\times [-1,1]$
which contains the singular set.
By looking at the intersection of $A$ with the disk,
we can find a boundary compressing disk for $A$.
By the irreducibility of the complement of the singular set of
$\check\PConway\times [-1,1]$, this implies that
$A$ is parallel to an annulus in $P'\subset P$.

Suppose next that $A$ is homeomorphic to $D^2(2,2)$.
Then $A$ is disjoint from $\tau_-$,
which has index $m\ge 3$.
Since $\partial A$ is parallel to $\partial P$
in $P\subset \partial\orbo(r;m)$,
we see as in the above that
$A$ is contained in
the suborbifold $(B^3_+, t(\infty)\cup\tau_+,w_+)$.
By looking at the intersection of $|A|$ with the disk $D_h$
as in the proof of Lemma ~\ref{lem:irreducible}(3),
we can see that $A$ is parallel to the suborbifold of $P$
bounded by $\partial A$.
}
\end{proof}

Since $\pi_1(\orbo(r;m))$ is not virtually abelian
by Lemma ~\ref{lem:hyperbolic},
we obtain the following proposition by Lemma ~\ref{lem:pared_orbifold}
and by the orbifold theorem for
Haken pared orbifolds \cite[Theorem ~8.3.9]{Boileau-Porti}.

\begin{proposition}
\label{prop:hyperbolic-pared-manifod}
For a rational number $r$ and an integer $m\ge 3$,
the pared orbifold $(\orbo(r;m),P)$ is hyperbolic,
i.e.,
there is a geometrically finite hyperbolic $3$-orbifold $M$
such that for some $\delta>0$ and $\mu$,
$(\orbo(r;m),P)$ is homeomorphic to
\[
(\mathrm{thick}_{\mu}(C_{\delta}(M)),\partial \mathrm{thick}_{\mu}(C_{\delta}(M))\cap \mathrm{thin}_{\mu}(C_{\delta}(M)),
\]
where $C_{\delta}(M)$ is the closed $\delta$-neighborhood of the convex core
$C(M)$ of $M$,
and $\mathrm{thick}_{\mu}(C_{\delta}(M))$ and $\mathrm{thin}_{\mu}(C_{\delta}(M))$ are
$\mu$-thick part and $\mu$-thin part.
Here $\mu$ is chosen so that $\mathrm{thin}_{\mu}(C_{\delta}(M))$ consists of only cuspidal part.
\end{proposition}

\begin{proof}[Proof of Theorem ~\ref{thm.Kleinian_heckoid}]
By the above proposition,
there is a faithful discrete representation
$\rho:\pi_1(\orbo(r;m))\to PSL(2,\CC)$
which maps the conjugacy class represented by the loop $\partial P$
to a parabolic transformation.
Recall that the Heckoid group
$\Hecke(r;n)=\pi_1(\orbs(r;n))$ is a subgroup of $\pi_1(\orbo(r;m))$
of index $2$ or $4$
by Proposition ~\ref{prop:even_Heckoid_group} and Definition ~\ref{def:odd_Heckoid_group}
and that it is generated by two elements in the conjugacy class of $\partial P$.
Hence, the restriction of $\rho$ to the subgroup $\Hecke(r;n)$
gives the desired isomorphism from $\Hecke(r;n)$ to a geometrically finite Kleinian group
generated by two parabolic transformations.
\end{proof}

At the end of this section,
we prove the following proposition,
which illustrates a significant difference
between odd Heckoid groups and even Heckoid groups.

\begin{proposition}
\label{prop-not-one-relator}
No odd Heckoid group is a one-relator group.
\end{proposition}

\begin{proof}
{\rm
Consider an odd Heckoid orbifold $\orbs(r;n)$.
By Proposition \ref{prop:odd-Heckoid-orbifold},
the singular set of $\orbs(r;n)$ has two or four
$1$-dimensional strata.
Note that
the above proof of Theorem ~\ref{thm.Kleinian_heckoid} shows that
$\orbs(r;n)$ is hyperbolic,
and so the interior of $\orbs(r;n)$ is homeomorphic to
a hyperbolic orbifold $\HH^3/\Gamma$,
where $\Gamma$ is a Kleinian group isomorphic to $\pi_1(\orbs(r;n))$.
Hence $\Gamma\cong \pi_1(\orbs(r;n))$ has two or four
conjugacy classes of maximal finite cyclic subgroups, accordingly.
On the other hand,
any one-relator group has a unique maximal finite cyclic subgroup up to conjugacy
(see \cite[Theorem ~IV.5.2]{lyndon_schupp}).
Hence the odd Heckoid group $\Hecke(r;n)=\pi_1(\orbs(r;n))$ cannot be a one-relator group.
}
\end{proof}

\section*{Acknowledgements}
The second author would like to thank Gerhard Burde for drawing
his attention to the work of Riley ~\cite{Riley2}
when he was visiting Frankfurt in 1997.
The second author would also like to thank Ian Agol
for sending him a copy of the slide of his talk ~\cite{Agol}.
Both authors would like to thank Yeonhee Jang,
Makoto Ozawa and Toshio Saito
for helpful discussions and information
concerning Section ~\ref{sec:Kleinian-Heckod-groups}.
They would also like to thank the referee
for very careful reading and helpful comments.

\bibstyle{plain}

\end{document}